\documentclass[11pt]{article}

\usepackage{latexsym, comment}
\usepackage{url}
\usepackage{epsfig,enumerate}
\usepackage{amsmath,amsthm,amssymb, bbm}
\usepackage[usenames]{color}\usepackage{graphicx}
\numberwithin{equation}{section}
\definecolor{brown}{cmyk}{0, 0.72, 1, 0.45}
\definecolor{grey}{gray}{0.5}

\newcounter{rot}%\addtocounter{rot}{1}, \therot
\def\leb{\leq_b}

\def\nn{\nonumber}
\def\a{\alpha} \def\b{\beta} \def\d{\delta} \def\D{\Delta}
\def\e{\epsilon}  \def\F{{\Phi}}  \def\g{\gamma}
\def\G{\Gamma} \def\i{\iota} \def\k{\kappa} 
     \def\l{\lambda}
\def\La{\Lambda} \def\m{\mu} \def\n{\nu} 
\def\r{\rho}  \def\s{\sigma} 
 \def\om{\omega}

\def\cD{{\cal D}}
\def\cP{{\cal P}}

\def\cC{{\cal C}}
\def\cB{{\cal B}}

\def\cV{{\cal V}}

\newtheorem*{conjecture*}{Conjecture}
\newtheorem{theorem}{Theorem}[section]
\newtheorem*{theorem*}{Theorem}

\newtheorem{lemma}[theorem]{Lemma}

\newtheorem{fact}[theorem]{Fact}

\newtheorem{corollary}[theorem]{Corollary}

\newtheorem{remark}[theorem]{Remark}

\newcommand{\brac}[1]{\left(#1\right)}

\newcommand{\bfrac}[2]{\left(\frac{#1}{#2}\right)}

\def\cE{{\cal E}}

\newcommand{\set}[1]{\left\{#1\right\}}
\def\sm{\setminus}
\def\seq{\subseteq}

\def\E{\mathbb{E}}

\def\P{\mathbb{P}}
\def\Pr{\mathbb{P}}

\def\cF{{\cal F}}

\newcommand{\bfo}{\mathbbm{1}}

\newcommand{\eps}{\epsilon}

\newcommand{\of}[1]{\left( #1 \right) }

\newcommand{\abs}[1]{\left| #1 \right|}

\newcommand{\sqbs}[1]{\left[ #1 \right]}

\renewcommand{\E}[1]{\mathbb{E}\sqbs{#1}}
\newcommand{\tbf}[1]{\textbf{#1}}
\renewcommand{\P}[1]{\mathbb{P}\left[ #1 \right]}
\renewcommand{\Pr}[1]{\mathbb{P}\left[ #1 \right]}

\allowdisplaybreaks[1]

\newcommand{\ignore}[1]{}

\newcommand{\cA}{{\cal A}}

\newcommand{\beq}[1]{\begin{equation}\label{#1}}
\newcommand{\eeq}{\end{equation}}

\newcommand{\Bin}{\operatorname{Bin}}

\DeclareMathOperator{\avg}{avg}
\DeclareMathOperator{\med}{med}
\DeclareMathOperator{\range}{range}

\def\cR{{\cal R}}

\def\w{{\bf w}}

\def\cH{{\cal H}}

\newcommand{\wiof}[3]{\textbf{w}_i\left(\{#1,#2\},#3\right)}
\newcommand{\wiofp}[3]{\textbf{w}_i\left((#1,#2),#3\right)}
\newcommand{\bv}{\tbf{v}}
\newcommand{\bz}{\tbf{z}}
\def\HH{\Upsilon}

\title{Rainbow Matchings and Hamilton Cycles in Random Graphs}

\author{
Deepak Bal\thanks{Department of Mathematics, Ryerson University, Toronto ON M5B 2K3, Canada}
\and 
Alan Frieze\thanks{Department of Mathematical Sciences, Carnegie Mellon University,
Pittsburgh PA15213, USA. Research supported in part by NSF grant ccf1013110.}
}

\begin{document}
\maketitle

\begin{abstract}
Let $HP_{n,m,k}$ be drawn uniformly from all $k$-uniform, $k$-partite hypergraphs where each part of the partition is a disjoint copy of $[n]$. We let $HP^{(\k)}_{n,m,k}$ be an edge colored version, where we color each edge randomly from one of $\k$ colors. We show that if $\k=n$ and $m=Kn\log n$ where $K$ is sufficiently large then w.h.p. there is a rainbow colored perfect matching. I.e. a perfect matching in which every edge has a different color.
We also show that if $n$ is even and $m=Kn\log n$ where $K$ is sufficiently large then w.h.p. there is a rainbow colored Hamilton cycle in $G^{(n)}_{n,m}$. Here $G^{(n)}_{n,m}$ denotes a random edge coloring of $G_{n,m}$ with $n$ colors. When $n$ is odd, our proof requires $m=\om(n\log n)$ for there to be a rainbow Hamilton cycle. 
\end{abstract}
\section{Introduction}
Given an edge-colored hypergraph, a set $S$ of edges is said to be {\em rainbow}
colored if every edge in $S$ has a diffent color. In this paper we
consider the existence of rainbow perfect matchings in $k$-uniform, $k$-partite hypergraphs and Hamilton
cycles in randomly colored random graphs. 

Let $U_1,U_2,\ldots,U_k$ denote $k$ disjoint sets of size $n$. Let $\cH\cP^{(\k)}_{n,m,k}$ denote the set of $k$-partite, $k$-uniform hypergraphs with vertex set $V = U_1 \cup U_2\cup \cdots \cup U_k$ and $m$ edges, each of which has been randomly colored with one of $\k$ colors. The random edge colored graph $HP^{(\k)}_{n,m,k}$ is sampled uniformly from $\cH\cP^{(\k)}_{n,m,k}$. 
 
In this paper we prove the following result
\begin{theorem}\label{th1}
There exists a constant $K$ such that if  $m \ge Kn\log n$ then 
$$\lim_{n\to\infty}\Pr{HP^{(n)}_{n,m,k}\text{ contains a rainbow perfect matching}}=1.$$
\end{theorem}
This result is best possible in terms of the number of colors $n$ and best possible up to a constant factor in terms of the number of edges.

We get the corresponding result for $k$-uniform hypergraphs $H^{(n)}_{kn,m,k}$ for free.
Here the edge set of $H^{(n)}_{kn,m,k}$ is a random element of $\binom{\binom{[kn]}{k}}{m}$ and each edge is randomly colored from $[n]$.
\begin{corollary}
If $m=Ln\log n$ and $L$ is sufficiently large then w.h.p. $H^{(n)}_{kn,m,k}$ contains a rainbow perfect matching. 
\end{corollary} 
\begin{proof}
We simply partition $[kn]$ randomly into $k$ sets of size $[n]$. Then we apply Theorem \ref{th1} with $K=L/k^k$.
\end{proof}

When $k=2$ the result of Theorem \ref{th1} can be expressed as follows:
\begin{corollary}
Let $A$ be an $n\times n$ matrix consructed as follows: Choose $Knlog n$ entries at random and give each a random integer from $[n]$. The remaining entries can be filled with zeroes. Then w.h.p. $A$ contains a latin transversal i.e. an $n\times n$ matrix $B$ with a single non-zero in each row and column, such that each $x\in [n]$ appears as a non-zero of $B$.
\end{corollary}

We can use Theorem \ref{th1} and a result of Janson and Wormald \cite{JW} to prove the following theorem on rainbow Hamilton cycles:
\begin{theorem}\label{th2}
There exists a constant $K$ such that if  $m \ge Kn\log n$ then with high probability, 
$$\lim_{\substack{n\to\infty\\ n\text {even}}}\Pr{G_{n,m;n}\text{ contains a rainbow Hamilton cycle}}=1.$$
When $n$ is odd we replace $m=Kn\log n$, by $m=\om n\log n$, where
$\om\to\infty$ arbitrarily slowly.
\end{theorem}
The result for $n$ odd is surely an artifact of the proof and we conjecture the same result is true for $n$ odd or even.

Previous results in this area have concentrated on the existence of rainbow Hamilton
cycles. For example, Frieze and Loh \cite{FL} showed that $G_{n,m;\k}$
contains a rainbow hamilton cycle w.h.p. whenever $m\sim \frac12n\log
n$\footnote{We write $A_n\sim B_n$ if $A_n=(1+o(1))B_n$ as
  $n\to\infty$} and $\k\sim n$. This result is asymptotically optimal
in number of edges and colors. Theorem \ref{th2} resolves a question
posed at the end of this paper (up to a constant factor) about the number of edges needed when
we have a minimum number of colors available. 
Perarnau and Serra \cite{PS13} showed that a random coloring of the
complete bipartite graph $K_{n,n}$ with $n$ colors contains a rainbow perfect matching.
Erd\H{o}s and Spencer \cite{ES91} proved the existence of a rainbow perfect matching in the complete bipartite graph $K_{n,n}$ when no color can be used more than $(n-1)/16$ times.
\section{Outline of the paper}
The proof of Theorem \ref{th1} is derived directly from the proof in the landmark paper of Johansson, Kahn and Vu \cite{JKV08}.  They prove something more general, but one of their main results concerns the ``Schmidt-Shamir'' problem, viz. how many random (hyper-)edges are needed for a 3-uniform hypergraph to contain a perfect matching. In this context, a perfect matching of a 3-uniform hypergraph $H$ on $n$ vertices $V$ is a set of $n/3$ edges that together partition $V$.

There is a fairly natural relationship between rainbow matchings of $k$-uniform hypergraphs and perfect matchings of $(k+1)$-uniform hypergraphs. This was already exploited in Frieze \cite{F1}. The basic idea is to treat an edge $\set{u_1,u_2,\ldots,u_k}$ of color $c\in C$ as an edge $\set{u_1,u_2,\ldots,u_k,c}$ in a $(k+1)$-uniform hypergraph $H$ with vertices $V\cup C$ and edges in $\binom{V}{k}\times C$. Then, assuming that $|V|=k|C|$ we ask for a perfect matching in $H$. Here we would take $V=[kn]$ and $|C|=n$ and construct $H$ randomly. The ``fly in the ointment'' so to speak, is that we cannot have two distinct edges $\set{u_1,u_2,\ldots,u_k,c_i},i=1,2$. This seems like a minor technicality and in some sense it is. We have not been able to find a simple way of resolving this technicality, other than modifying the proof in \cite{JKV08}.

We slightly sharpen our focus and consider multi-partite hypergraphs.
Let $K_{n,k}$ be the complete $k$-partite, $k$-uniform hypergraph where each part has $n$ vertices. 
Its vertex set $V$ is the union of $k$ disjoint sets $U_1\cup U_2\cup\cdots\cup U_k$, each of size $n$.  We let the edge set of $K_{n,k}$ be $\cV = \cV_k = U_1\times\cdots\times U_k$. $HP_{n,m,k}$ is obtained by choosing $m$ random edges from $\cV$.

Our approach, taken from \cite{JKV08}, is to start with a random coloring of the complete $k$-partite hypergraph $K_{n,k}$. Denote this edge colored graph by $K^{(n)}_{n,k}$. We show in Section \ref{norm} that w.h.p. $K^{(n)}_{n,k}$ has a large number of rainbow perfect matchings. We then randomly delete edges one by one showing that w.h.p. the remaining graph $H_i$, after $i$ steps, still contains many rainbow perfect matchings. Here we need $i\leq N-Kn\log n$ where $N=n^k$ and $K$ is sufficiently large.

We let $\F_i$ denote the number of rainbow perfect matchings in $H_i$ and consider $\xi_i=1-\frac{\F_i}{\F_{i-1}}$. If we can control the sequence $(\xi_i)$ then we can control the number of rainbow perfect matchings in $H_i$. It is enough to control
$S_i=\sum_i\xi_i$. We will let $\w_i(e)$ denote the number of rainbow perfect matchings that contain a particular edge $e\in E_i$, the edge-set of $H_i$.  $S_i$ will be concentrated around its mean if we show that w.h.p. the maximum value of $\w_i(e)$ is only $O(1)$ times the average value of $\w_i(e)$ over $e\in E_i$. This is the event $\cB_i$ defined in \eqref{cBi}. Proving that $\cB_i$ occurs w.h.p. is the heart of the proof. 

In Section \ref{av2m} there is a switch from bounding the ratio of max to average to bounding the ratio of max to median. It is then shown that it is unlikely for the maximum to be more than twice the median. Entropy and symmetry play a significant role here and it is perhaps best to leave the reader to enjoy this clever set of ideas from \cite{JKV08} when he/she gets to them.

Once we have Theorem \ref{th1}, it is fairly straightforward to use the result of \cite{JW} to obtain Theorem \ref{th2}. This is done in Section \ref{Ham}.
\section{The number of rainbow perfect matchings in $K^{(n)}_{n,k}$}\label{norm}
To begin, we will show that the number of rainbow perfect matchings in $K^{(n)}_{n,k}$, with its edges randomly colored by $n$ colors is concentrated around its expected value.  
\begin{lemma}\label{startinglem} Let $\F(K^{(n)}_{n,k})$ represent the number of rainbow matchings of $K^{(n)}_{n,k}$. Then w.h.p.,
\[\F(K^{(n)}_{n,k}) \sim\frac{(n!)^{k}}{n^n}.\]
\end{lemma}
\begin{proof}
Let $X$ be a random variable representing the number of rainbow matchings in $K^{(n)}_{n,k}$.
Then there are $(n!)^{k-1}$ distinct perfect matchings and each has probability $\frac{n!}{n^n}$ of being rainbow colored. Hence,
\begin{equation}\label{EX}
\E{X} = (n!)^{k-1}\times \frac{n!}{n^n}=\frac{(n!)^{k}}{n^n}.
\end{equation}
We use Chebyshev's Inequality to show that $X$ is concentrated around this value. It  is enough to show that 
\[\E{X^2} \le (1+o(1))\E{X}^2.\]

Given a fixed matching $M$ with $\ell$ edges, let $N_{\ell}$ represent the number of matchings covering the same vertex set as $M$ but are edge disjoint from $M$. 
Then inclusion-exclusion gives
\begin{align*}
N_\ell &= \sum_{i=0}^{\ell}(-1)^i\binom{\ell}{i}((\ell-i)!)^{k-1} \\
&=(\ell!)^{k-1}\sum_{i=0}^{\ell}\frac{(-1)^i}{i!}\bfrac{(\ell-i)!}{\ell!}^{k-2}.
\end{align*}
Now, suppose we have an integer sequence $\l = o(\sqrt{\ell})$ and $\l \to \infty$ with $\ell$. Then the Bonferroni inequalities tell us that
\begin{equation}
(\ell!)^{k-1}\sum_{i=0}^{2\l-1}\frac{(-1)^i}{i!\ell^{i(k-2)}}\of{1 +o(1)}\le N_\ell \le (\ell!)^{k-1}\sum_{i=0}^{2\l}\frac{(-1)^i}{i!\ell^{i(k-2)}}\of{1 +o(1)}.
\end{equation}
So as long as $\ell \to \infty$,
\[N_\ell = (\ell!)^{k-1}\of{e^{-1}1_{k=2}+1_{k\geq 3} + o(1)}.\]
Then
we have
\begin{align}
\E{X^2} &= \sum_{\ell=0}^{n}(n!)^{k-1}\binom{n}{\ell}N_{n-\ell}\frac{n!}{n^n}\frac{(n-\ell)!}{n^{n-\ell}}\nn\\
&=\E{X}\sum_{\ell=0}^{n}\frac{n!}{\ell!n^{n-\ell}}N_{n-\ell}\nn\\
&=\E{X}\sum_{\ell=0}^{\log n}\frac{n!}{\ell!n^{n-\ell}} ((n-\ell)!)^{k-1}
\of{e^{-1}1_{k=2}+1_{k\geq 3} + o(1)}\label{eq:mainterm}\\
&\quad+\E{X}\sum_{\ell=\log n}^{n}\frac{n!}{\ell!n^{n-\ell}}N_{n-\ell}\label{eq:smallterm}
\end{align}
We now bound \eqref{eq:mainterm} and \eqref{eq:smallterm} in turn.
We have that \eqref{eq:mainterm} is equal to
\begin{align*}
 &\E{X}^2\of{e^{-1}1_{k=2}+1_{k\geq 3} + o(1)}\sum_{\ell=0}^{\log n}\frac{1+o(1)}{\ell!n^{\ell(k-2)}}\\
&=\E{X}^2\of{e^{-1}1_{k=2}+1_{k\geq 3} + o(1)}\of{e1_{k=2}+1_{k\geq 3} + o(1)} = \E{X}^2(1+o(1))
\end{align*}
It remains to show that \eqref{eq:smallterm} is $o(\E{X}^2)$. We split this sum into 2 parts. First, using the trivial bound on $N_{n-\ell}\leq ((n-\ell)!)^{k-1}$, we have
\begin{align}
 \bfrac{1}{\E{X}^2}\E{X}\sum_{\ell=\log n}^{n-\log n}\frac{n!}{\ell!n^{n-\ell}}N_{n-\ell}&=\sum_{\ell=\log n}^{n-\log n}\frac{n^\ell}{\ell!(n!)^{k-1}}N_{n-\ell}\nn\\
&\le\sum_{\ell=\log n}^{n-\log n}\frac{n^\ell}{\ell!(n!)^{k-1}}((n-\ell)!)^{k-1}\label{214}.
\end{align}
Since in this range, both $\ell$ and $n-\ell$ approach infinity with $n$, we may apply Stirling's approximation to all factorials to get that for some constant $c$, \eqref{214} is at most
\begin{align*}
&c \sum_{\ell=\log n}^{n-\log n}n^\ell\cdot\frac{e^\ell}{\ell^{\ell+1/2}}\cdot\frac{e^{(k-1)n}}{n^{(k-1)(n+1/2)}}\cdot \frac{(n-\ell)^{(k-1)(n-\ell+1/2)}}{e^{(k-1)(n-\ell)}} \\
&\le c\cdot\sum_{\ell=\log n}^{n-\log n}\bfrac{e^k}{\ell n^{k-2}}^\ell \\
&\le cn\bfrac{e^k}{\log n}^{\log n} = o(1).
\end{align*}
For $\ell\geq n-\log n$ we bound $N_{n-\ell}\leq \binom{n}{\log n}((\log n)!)^{k-1}$ and then we have that for some constant $c'$, \eqref{214} is at most
$$\sum_{\ell=n-\log n}^{n}\frac{n^\ell}{\ell!(n!)^{k-1}}\binom{n}{\log n}((\log n)!)^{k-1}\le c'\log n \cdot \frac{e^n\cdot 2^n\cdot (\log n)^{k\log n}}{(n-\log n)!} = o(1).$$
\end{proof}
This completes the proof of Lemma \ref{startinglem}.

We will need the Chernoff bounds:
\begin{fact}\label{chernoff}
 Let $X$ be the sum of independent Bernoulli random variables and let $\E{X}=\m$. Then
\begin{align*}
\Pr{\abs{X-\m} >\eps \m} &\le 2e^{-\e^2\m/3}\qquad 0\leq\e\leq 1.\\
\Pr{X\geq \a \m}&\leq \bfrac{e}{\a}^{\a \m} \qquad \a>e.
\end{align*}
\end{fact}

%\end{comment}
\section{Proof of Theorem \ref{th1}}

Let the color set be $C$ (so $\abs{C} = n$) and let $\i:E(K_{n,k}) \to C$ be the random coloring of the edges. 
Let 
$$e_1,\ldots, e_N,\ N = n^k$$ 
be a random ordering of the edges of $K_{n,k}^{(n)}$, where we have used $\i$ to color the edges of $K_{n,k}$. Let $H_i = K_{n,k}^{(n)} - \set{e_1,\ldots,e_i}=(V,E_i)$. Here if $H=(V,E)$ is a hypergraph and $A\subseteq E,S\subseteq V,D\subseteq C$ then $H-A-S-D$ is the hypergraph on vertex set $V\setminus S$ with those edges in $E\setminus A$ that are disjoint from $S$ and do not use a color from $D$.

For a color $c\in C$, let $cd_{H_i}(c)=\abs{\set{e\in E_i: \i(e)=c}}$ be the number of edges of $H_i$ that have color $c$. 
\subsection{Tracking the number of rainbow matchings}
For an edge-colored hypergraph $H$, we let $\cF(H)$ denote the set of rainbow  perfect matchings of $H$ and we let 
$$\F(H)=|\cF(H)|.$$
Let $\cF_t=\cF(H_t)$ and $\F_t=|\cF_t|$ and then
$$\F_t=\F_0\frac{\F_1}{\F_0}\cdots\frac{\F_t}{\F_{t-1}}=\F_0
(1-\xi_1)\cdots(1-\xi_t)$$
or
$$\log\F_t=\log\F_0+\sum_{i=1}^t\log(1-\xi_i).$$
where, by Lemma \ref{startinglem}, we have that w.h.p.
\beq{10}
\log \F_0=\log \frac{(n!)^k}{n^{n}}(1+o(1))=(k-1) n\log n-c_1n,
\eeq
where 
\beq{c1}
0<c_1<k+1.
\eeq
We also have
\beq{11}
\E{\xi_i}=\g_i=\frac{n}{N-i+1}\leq \frac{1}{K\log n}.
\eeq
for $i\leq T=N-Kn\log n$.

Equation \eqref{11} becomes, with 
$$p_t=\frac{N-t}{N},$$
\beq{13}
\sum_{i=1}^t\E {\xi_i}=\sum_{i=1}^t\g_i=n
\brac{\log\frac{N}{N-t}+O\bfrac{1}{N-t}}=
n\brac{\log\frac{1}{p_t}+O\bfrac{1}{N-t}}
\eeq
using the fact that $\sum_{i=1}^N\frac{1}{i}=\log N+(Euler's\ constant)+O(1/N)$.

For $t=T$ this will give 
$$p_T=\frac{Kn\log n}{N}$$ 
and so for $t\leq T$ we have
\beq{c2}
\sum_{i=1}^t\g_i=-n \log p_t+o(n)\leq (k-1)n\log n.
\eeq

Our basic goal is to prove that if we define
$$\cA_t=\set{\log\F_t>\log\F_0-\sum_{i=1}^t\g_i-(c_1+1)n },$$
then
\beq{At}
\Pr{\bar{\cA}_t}\leq n^{-K^{1/3}/5}\text{ for }t\leq T.
\eeq
Given that we can make $K$ as large as we like, this implies Theorem \ref{th1}.

\subsection{Important properties}
We now define some properties that will be used in the proof.

If $e = (x_1,\ldots,x_k)$ and $c\in C$ then $\w_i(e,c)$ is the number of rainbow matchings of $H_i-\set{x_1,\ldots,x_k}$ that do not use an edge of color $c$. In particular if $e$ is an edge, then $\w_i(e,\i(e))$ is the number of rainbow matchings of $H_i$ which use the edge $e$.
We will usually shorten $\w_i(e,\i(e))$ to $\w_i(e)$ for $e\in E_i$.

In the following we have 
$$\w_i(E_i)=\sum_{e\in E_i}\w_i(e)\text{ and }\avg_{e\in E_i}\w_i(e)=\frac{\w_i(E_i)}{|E_i|}.$$

Let
\beq{cBi}
\cB_i = \set{\frac{\max_{e\in E_i} \w_i(e)}{\avg_{e\in E_i}\w_i(e)} \le L=K^{1/2}}
\eeq

\[\cR_i = \begin{Bmatrix}
\forall v\in V,\,\abs{d_{H_i}(v) - n^{k-1}p_i} \leq \e_1n^{k-1}p_i\\
\textrm{ and }\\
\forall c\in C,\, \abs{cd_{H_i}(c)-n^{k-1}p_i} \leq \e_1n^{k-1}p_i
\end{Bmatrix}
\]
where $\e_1=\frac{1}{K^{1/3}}$.

It will take most of the paper to show that $\cB_i$ occurs w.h.p. for all $i\leq T$, but $\cR_i$ is easily dealt with.
\subsubsection{Dealing with $\cR_i$}
First, we observe that $H_i$ is distributed as $HP^{(n)}_{n,N-i,k}$ and so for any hypergraph property $\cP$ we can write
\beq{GP}
\Pr{H_i\in \cP}\leb m\Pr{HP^{(n)}_{n,p_i,k}\in \cP},
\eeq
where $HP^{(n)}_{n,p_i,k}$ is the corresponding independent model in which each possible edge is included with probability $p_i$. 
This follows from $\Pr{HP^{(n)}_{n,p,k}\in \cP}\geq \binom{N}{m}p^m(1-p)^{N-m}$ where $p=m/N$. The notation $A\leb B$ is a substitute for $A=O(B)$.

Applying the Chernoff bound we see that for any $v,i$ we have 
\beq{eq1}
\Pr{|d_{H_i}(v)-n^{k-1}p_i|\geq \e_1n^{k-1}p_i}\leq 2ne^{-\e_1^2n^{k-1}p_i/3}\leq n^{-K^{1/3}/4}.
\eeq
For a fixed color $c$ we see that $cd_{H_i}(c)$ is distributed as the binomial $\Bin(N-i,1/n)$ which has expectation $n^{k-1}p_i$.
Applying the Chernoff bound once more we see then that for a fixed color $c$ we have
\beq{eq2}
\Pr{\abs{cd_{H_i}(c)-n^{k-1}p_i} \geq \e_1n^{k-1}p_i}\leq 2e^{-\e_1^2n^{k-1}p_i/3}\leq n^{-K^{1/3}/4}.
\eeq
This deals with $\cR_i,i\leq T$.

We now consider the first time $t\leq T$, if any, where $\cA_t$ fails. Then,
$$\bar{\cA}_t\cap\bigcap_{i<t}\cA_i\subseteq \sqbs{\bigcup_{i<t}\bar{\cR}_i}\cup\sqbs{\bigcup_{i<t}\cA_i\cR_i\bar{\cB}_i}
\cup\sqbs{\bar{\cA}_t\cap\bigcap_{i<t}(\cB_i\cR_i)}$$

We can therefore write
\beq{aim}
\P{\bar{\cA}_t\cap\bigcap_{i<t}\cA_i}<\sum_{i<t}\P{\bar{\cR}_i}+
\sum_{i<t}\P{\cA_i\cR_i\bar{\cB}_i}+
\P{\bar{\cA}_t\cap\bigcap_{i<t}(\cB_i\cR_i)}.
\eeq

\subsection{Concentration of the number of rainbow matchings}
We define
$$\cE_i=\set{\cB_j, \cR_j,j<i}.$$
We will first show that 
\beq{eq3}
\cE_i \implies \xi_i \leq \frac{1}{K^{1/2}\log n}.
\eeq
First we have
\begin{align*}
\w_{i-1}(E_{i-1})&= \sum_{e\in E_{i-1}}\sum_{F\in \cF_{i-1}}\bfo_{e\in F}\\
&=n\F_{i-1}.
\end{align*}
So for any $f\in E_{i-1}$,
\begin{align*}
\F_{i-1}&=\frac{1}{n}\w_{i-1}(E_{i-1})\\
&\ge \frac{1}{Ln}\abs{E_{i-1}}\max_{e\in E_{i-1}}\w_{i-1}(e)\\
&\ge \frac{N}{Ln}p_{i-1}\w_{i-1}(f).
\end{align*}
Hence, if the event $\cE_i$ holds then
$$\xi_i \le \max_{e\in E_{i-1}}\frac{\w_{i-1}(e)}{\F_{i-1}} \le \frac{Ln}{Np_{i-1}} \le \frac{L}{K\log n} \leq\frac{1}{K^{1/2}\log n},$$
confirming \eqref{eq3}.

Now define
\[Z_i =	\begin{cases}\xi_i - \g_i &\text{if $\cE_i$  holds}\\
0&\text{otherwise}
\end{cases}
\] 
and let
\[X_t = \sum_{i=1}^{t} Z_i.\]
We will show momentarily that 
\beq{Xt}
\Pr{X_t \ge n} \le e^{-\Omega(n)}.
\eeq
So if we do have $\cE_t$ for $t \le T$ (so that $X_t = \sum_{i=1}^{t}(\xi_i - \g_i)$) and $X_t \le n$ then
\[\sum_{i=1}^t \xi_i < \sum_{i=1}^t \g_i + n \le (k-1)n\log n+n\]
and so
\[\sum_{i=1}^t\xi_i^2 \le \frac{1}{K^{1/2}\log n}\cdot\sum_{i=1}^t\xi_i
\leq 2K^{-1/2}n.\]

So,
$$\log\F_t>\log\F_0-\sum_{i=1}^t(\xi_i+\xi_i^2)>\log\F_0-\sum_{i=1}^t\g_i-2n.$$
This deals with the third term in \eqref{aim}.
(If $\cE_t$ holds then $\cA_t$ holds with sufficient
probability).

Let us now verify \eqref{Xt}. Note that $|Z_i|\leq \frac{1}{K^{1/2}\log n}$ and that for any $h>0$
\beq{25}
\Pr{X_t\geq n}=\Pr{e^{h(Z_1+\cdots+Z_t)}\geq e^{hn}}\leq \E{e^{h(Z_1+\cdots+Z_t)}}e^{-hn}
\eeq
Now $Z_i=\xi_i-\g_i$ (whenever $\cE_i$ holds) and
$\E{\xi_i\mid\cE_i}=\g_i$. The conditioning does not affect the
expectation since we have the same expectation given any previous
history. 
Also $0\leq \xi_i\leq \e=\frac{1}{\log n}$ (whenever $\cE_i$ holds). 
So, with $h\leq 1$, by convexity
\begin{multline*}
\E{e^{hZ_i}}=\E{e^{hZ_i}\mid \cE_i}\Pr{\cE_i}+\E{e^{hZ_i}\mid\neg\cE_i}\Pr{\neg\cE_i}\leq\\
 e^{-h\g_i}\E{1-\frac{\xi_i}{\e}+\frac{\xi_i}{\e}e^{h\e}\bigg|\cE_i}\Pr{\cE_i}+\Pr{\neg\cE_i}\\
=e^{-h\g_i}\brac{1-\frac{\g_i}{\e}+\frac{\g_i}{\e}e^{h\e}}\Pr{\cE_i}+1-\Pr{\cE_i}\leq e^{h^2\e\g_i}.
\end{multline*}
So,
$$\E{e^{h(Z_1+\cdots+Z_t)}}\leq e^{h^2\e\sum_{i=1}^t\g_i}$$
and going back to \eqref{25} we get
$$\Pr{X_t\geq n}\leq e^{h^2\e\sum_{i=1}^t\g_i-hn}.$$
Now $\sum_{i=1}^t\g_i=O(n\log n)$ and so putting $h$
equal to a small enough positive constant makes the RHS of the above
less than $e^{-hn/2}$ and \eqref{Xt} follows.

\subsection{From average to median}\label{av2m}
If $I\subset[k]$, we write $\cV_I$ for the collection of
$\abs{I}$-sets of vertices using exactly one vertex from each of $U_i,
i\in I.$ For $r\le k$, we let $\cV_r=\bigcup_{|I|=r}\cV_I$.
Given $\bv \in \cV_{r}$, we define $I(\bv)$ by $\bv\in \cV_{I(\bv)}$ and $I^c(\bv) = [k]\sm I(\bv)$. 

Now for a multi-set $X\subseteq \mathbf{R}$ we let $\med{X}$, the median of
$X$, be the largest value $x\in X$ such that there are at least
$|X|/2$ elements of $X$ that are larger than $x$. Then define 
\[
\cC_i= \begin{Bmatrix}
\forall \bv\in \cV_{k-1}, c\in C, \max_{w\in U_{I^c(\bv)}}\w_i(\of{\bv,w},c) \le \max\set{\frac{\F_i}{2^kN}, 2\med\displaylimits_{u\in U_{I^c(\bv)}}\w_i(\of{\tbf v,u},c) }\\
\textrm{ and } \\
 \forall \tbf{v}\in \cV_k, \max_{c\in C}\w_i(\tbf v,c) \le \max\set{\frac{\F_i}{2^kN}, 2\med\displaylimits_{c\in C}\w_i(\bv,c) }.
\end{Bmatrix}
\]

We will prove
\begin{align}
\Pr{\cR_i\cC_i\bar{\cB_i}} &< n^{-K^{1/3}/4} \label{rcb}\\
\Pr{\cA_i\cR_i\bar{\cC_i}} & < n^{-K^{1/3}/4}\label{arc}.
\end{align}

Note that \eqref{rcb} and \eqref{arc} imply that
$$\Pr{\cA_i\cR_i\bar{\cB}_i}=\Pr{\cA_i\cR_i\bar{\cB}_i\cC_i}+\Pr{\cA_i\cR_i\bar{\cB}_i\bar{\cC}_i}\leq 2n^{-K^{1/3}/4}.$$
This deals with the middle term in \eqref{aim}.

\subsection{Proof of \eqref{rcb}}\label{secrcb}
First, we suppose that 
\beq{least}
\P{\cR_i\cC_i} \ge n^{-K^{1/3}/4},
\eeq 
otherwise \eqref{rcb} holds trivially.
For $\tbf{v}\in \cV_{k-1}$ and $c\in C$, we let $\psi_V(\bv,c) = \max_{w\in \cV_{I^c(\bv)}}\w\of{\of{\bv , w},c}$ and for $\bv\in \cV_k$, we let $\psi_C(\bv) = \max_{c\in C}\w(\bv,c).$ Let 
\beq{vwc}
\psi_0=\w(\bv',c')=\max_{\bv\in \cV_k}\max_{c\in C}\w(\bv,c).
\eeq
\begin{lemma}\label{RCBlemma}
Suppose that $B$ is such that $\psi_0 \ge 2^kB$ and that for each $\bv\in \cV_{k-1}, c\in C$ with $\psi_V(\bv,c) \ge B$, we have
\begin{equation}\label{psiV}\abs{\set{w\in \cV_{I^c(\bv)} : \w(\of{\bv,w},c)\ge \frac{1}{2}\psi_V(\bv,c)}} \ge \frac{n}{2}\end{equation}
and for all $\bv\in V_k$ with $\psi_C(\bv)\ge B$, we have
\begin{equation}\label{psiC}\abs{\set{c\in C : \w(\bv,c)\ge \frac{1}{2}\psi_C(\bv)}} \ge \frac{n}{2}.\end{equation}
Then we have
\begin{equation}\label{psiES}
\abs{\set{(\bv,c)\in \cV_k \times C\,:\,\w(\bv,c) \ge \frac{\psi_0}{2^{k+1}}}} \ge \frac{n^{k+1}}{2^{k+1}}
\end{equation}
\end{lemma}
 \begin{proof}
Suppose $\bv' = (v_1', \ldots, v_k'),c'$ are as in \eqref{vwc}. Then by \eqref{psiV}, there there
is a set $W_1\subset U_1$ of size at least $n/2$ such that if $w_1\in W_1$ then
$\w((w_1,v_2',\ldots,v_k'),c') \ge \frac{1}{2}\psi_V((v_2',\ldots,v_k'),c') =\frac12\psi_0\geq 2^{k-1}B$. For each
$w_1\in W_1$, since we have $\psi_V((w_1,v_3',\ldots,v_k')),c') \ge \frac12\psi_0\geq2^{k-1}B$, we may apply
\eqref{psiV} once more to find a set $W_2^{w_1}\subseteq U_2$ of size at least $n/2$
such that if $w_2\in W_2^{w_1}$ then $\w((w_1,w_2,v_3',\ldots,v_k'),c') \ge \frac12\psi_V((w_1,v_3',\ldots,v_k')),c')\ge \frac14\psi_0\ge 2^{k-2}B$. 

Continuing in this way, for every choice of $w_1\in W_1$, $w_2 \in W_2^{w_1}$, $w_3\in W_3^{w_1,w_2},\ldots,w_k\in W_k^{w_1,\ldots w_{k-1}}\subseteq U_k$, we have
$\w((w_1,\ldots,w_k),c')\ge \frac{1}{2^k}\psi_0 \ge B$. Thus every such choice of $w_1,\ldots, w_k$, we have
 $\psi_C((w_1,\ldots,w_k)) \ge\frac{1}{2^k}\psi_0\geq B$, so to
finish, we apply \eqref{psiC} to find a set $D^{w_1,\ldots,w_k}\subseteq C$ of size at
least $n/2$ such that if $d\in D^{w_1,\ldots,w_k}$ then 
\(\w((w_1,\ldots,w_k),d) \ge \frac{\psi_0}{2^{k+1}}.\) Since there are $n/2$ choices for vertices in each part and $n/2$ choices for colors, we have that the number of choices total is at least $\frac{n^{k+1}}{2^{k+1}}$ as desired.
\end{proof}
For $\bv \in \cV_k$, let $H_i^{\bv c}$ be the hypergraph $H_i$ with vertices in $\bv$ removed as well as all edges with color $c$.
Now let $\w_i(\bv, c)$ be the number of rainbow matchings in $H_i^{\bv c}$. 
Suppose that $\cC_i$ holds and let $B=\frac{\F_i}{2^kN}$. 
Note that $\psi_0\geq 2^kB$ else we would have $\psi_0 < \frac{\F_i}{N} < \avg_{e\in E_i}\w_i(e)$, contradiction.

So for all $\bv\in \cV_{k-1}$, $c\in C$ with $\psi_V(\bv,c) \ge B$, we have \[\max_{w\in \cV_{I^c(\bv)}}\w_i\of{(\bv,w),c} \le 2\med_{w\in \cV_{I^c(\bv)}}\w_i(\of{\bv,w},c).\]
This is condition \eqref{psiV}. Similarly, the second condition of $\cC_i$ gives us \eqref{psiC}. So we may conclude that
\begin{equation}\label{cubicnumber}
\abs{\set{(\bv,c)\in \cV_k\times C\,:\, \w_i(\bv, c) \ge \frac{1}{2^{k+1}}\psi_0}} = \frac{n^{k+1}}{2^{k+1}}.
\end{equation}

Let 
\begin{equation}\label{EStarDef}
 E_i^* := \set{e\in E_i\,:\, \w_i(e) \ge \frac{1}{2^{k+2}}\max_{e\in E_i}\w_i(e)}.
\end{equation} 
We will show that 
\begin{equation}\label{Estarbig}
\Pr{\abs{E_i^*} \le \frac{Np_i}{2^{2k+7}} \middle\vert \cR_i\cC_i}\le n^{-K^{1/3}}.
\end{equation}

Let $\g=\frac{1}{2^{k+3}}$. By equation \eqref{cubicnumber} there are $\g n$ vertices
in $X_1\subseteq U_1$ such that if $x_1\in X_1$ then there are $\g n$
choices for $c_1 \in C_1(x_1)\subseteq C$ such that there are $\g
n^{k-1}$ choices for  $\tbf{x} = (x_2,\ldots,x_k)\in
U_2\times\cdots\times U_k$, such that if $x_1\in X_1,\,c_1\in C_1(x_1)$ then
\begin{equation}\label{x1x2c1big}
\w_i(\of{x_1,\tbf{x}},c_1) > \frac{1}{2^{k+1}}\psi_0 \ge \frac{1}{2^{k+1}}\max_{e\in E_i}\w_i(e).
\end{equation}
Now fix $0\leq \ell\leq 2kn\log n$ and let $\La=2^\ell$. Fix a vertex
$x_1\in X_1$ and let 
$$A_\La(x_1)=\set{\tbf{x}\in \cV_{[2,k]},c_1\in C: \w_i(\of{x_1,\tbf{x}},c_1)\geq \La}$$
and let 
$$B_\La(x_1)=\set{\tbf{x}\in \cV_{[2,k]},c_1\in C: c_1=\i(x_1,\tbf{x})\text{ and }\w_i(\of{x_1,\tbf{x}},c_1)\geq \La}$$
Here $\La$ will be
an approximation to the random variable $\psi_0/2^{k+1}$. Using $\La$ in
place of $\psi_0/2^{k+1}$ reduces the conditioning. There are not too
many choices for $\Lambda$ and so we will be able to use the union bound
over $\La$.

Let $S,T$ denote disjoint subsets of $\set{x_1}\times \cV_{[2,k]}\times C$.
Note that without the conditioning $\cR_i\cC_i$ the event
$\set{S\subseteq A_\La,T\cap A_\La=\emptyset}$ will be independent of the
event 
\beq{event}
\bigcap_{(e,c)\in S}\set{e\in E_i,\i(e)=c}\cap\bigcap_{(e,c)\in
  T}\neg\set{e\in E_i,\i(e)=c}.
\eeq 
This is because $\w_i(\of{x_1,\tbf{x}},c_1)$ depends only on the
existence and color of edges $f$ where if $\tbf{x}=(x_2,x_3,\ldots,x_k)$,
$$\set{x_1,x_2,\ldots,x_k}\cap f=\emptyset.$$
If we work with the model $HP_{n,k,p_i}$ in place of $H_i$, without the 
conditioning, then $\E{|B_\La(x_1)|}=|A_\La|p_i/n$. Also, we can express $|B_\La(x_1)|$ as the sum of independent Bernoulli random variables, one for each possible value of $\tbf{x}$. The variable $Z$ corresponding to a fixed $\tbf{x}$ will be one iff there is a $c_1\in C$ such that $((x_1,\tbf{x}),c_1)\in B_\La(x_1)$. 

Hence, if $|A_\La(x_1)|\geq \D=\g^2 N$, then using Fact
\ref{chernoff} and equations \eqref{GP} and \eqref{least},
$$\Pr{|B_\La(x_1)|\leq \D p_i/2n\mid \cR_i\cC_i}\leq \frac{\Pr{|B_\La(x_1)|\leq \D p_i/2n}}{\Pr{\cR_i\cC_i}}\leq n^{K^{1/3}/4}e^{-\D p_i/12n}\leq n^{-\g^2K/20}.
$$
There are at most $n$ choices for $x_1$. The number of choices for
$\ell$ is $2kn\log n$ and for one of these we will have 
$2^\ell\leq \frac{1}{2^{k+1}}\max\w_i(E_i)\le 2^{\ell+1}$ and so with probability
$1-n^{2+o(1)-\g^2K/20}$ we have that for each choice of $x_1\in X_1$
there are $\g^2Np_i/2$ choices for $\tbf{x},c$ such that
$(e=(x_1,\tbf{x}),c=\i(x_1,\tbf{x}))\in B_\La(x_1)$ and
$\w_i(e,c)>\frac{1}{2^{k+2}}\max\w_i(E_i)$. 
Observe that we have $2^{k+2}$ in in place of $2^{k+1}$, because we will want the above to hold for a value of $\La$ where $\La\leq \max\w_i(E_i)\leq 2\La$.
This verifies \eqref{Estarbig} and we have
$$\frac{\sum_{e\in E_i}\w_i(e)}{\max\w_i(E_i)}\geq 
\frac{\sum_{e\in {E_i^*}}\w_i(e)}{\max\w_i(E_i)}
\geq \frac{|{E_i^*}|}{2^{k+2}}\geq \frac{Np_i}{2^{3k+9}}\geq \frac{|E_i|}{2^{3k+10}}$$
which implies property $\cB_i$ if $K$ is sufficiently large.\\

 \subsection{Proof of \eqref{arc}}
 Recall that for a discrete random variable $X$, the (base $e$) entropy $H(X)$, is defined by
 \[H(X) = \sum_{x}p_x\log\of{\frac{1}{p_x}}\]
where the sum ranges over possible values of $X$ and $p_x=\P{X=x}.$  
%The following Lemma appears in a  slightly different form in \cite{CGFS}.
% \begin{lemma}[Shearer]\label{shearer}
%Let $S$ be a finite set and let $A_1,\ldots,A_m$ be subsets of $S$
%such that every element of $S$ is contained in at least $k$ of
%$A_1,\ldots,A_m$. Let $\scr{F}$ be a collection of subsets of $S$ and
%let $\cF_i = \set{F\cap A_i : F\in\cF}$ for $1\le i\le m$. Further,
%let $X$ be a uniformly chosen random element of $\cF$ and let $X_i = X\cap A_i$ be the associated random variable %taking values in $\cF_i$. Then
%\begin{equation}
% H(X) \le \frac{1}{k}\sum_{i=1}^{m}H(X_i).
%\end{equation}
%\end{lemma}
 
 The following lemma is proved in \cite{JKV08}.
 \begin{lemma}\label{JKV6.2}
  If $H(X) > 
  \log\of{\abs{S}} - M,M=O(1)$, then there are $a,b \in\range(\w)$ with $a\le b \le \r_Ma$ such that for $J = \w^{-1}\sqbs{a,b}$ we have
  \[\abs{J} \geq \s_M\abs{S}\] and \[\w(J) > 0.7\w(S).\]
Here we can take $\r_M=2^{4(M+\log 3)}$ and $\s_m=2^{-2M-2}$.
 \end{lemma}

 To prove \eqref{arc}, assume that we have $\cA_i$ and $\cR_i$ and that $\cC_i$ fails. Then we have two cases.

 Suppose $\bv \in \cV_{k-1}$, $x \in \cV_{I^c(\bv)}$, and $c\in C$. Let $H_i^{\bv xc}$ be the sub-graph of $H_i$ induced by
 $V\setminus\{\bv,x\}$ where all edges of color $c$ have been deleted.

 \subsubsection{Case 1} 
Suppose that $\cC_i$ fails because there exists $\bv  \in \cV_{k-1}$ and $c\in C$ such that \[\max_{\xi\in \cV_{I^c(\bv)}}\w_i((\bv,\xi),c)  > \max\set{\frac{\F_i}{2^kN}, 2\med_{\xi\in \cV_{I^c(\bv)}}\w_i((\bv,\xi),c)}.\] Let $x$ be the value of $\xi$ which maximizes $\w_i((\bv,\xi),c)$.
 For ease of notation, let us suppose that $\bv = (v_1,\ldots, v_{k-1}) \in U_1\times\cdots\times U_{k-1}$, so that $I^c(\bv) = \set{k}$. Let $y\in U_{k}\sm \set{x}$ be a vertex with 
\beq{EQ1} 
\w_i((\bv,y),c) \le \med_{x}\wiofp{\bv}{x}{c}
\eeq
 and \[h(y,H_i^{\bv xc}):= H(X(y,H_i^{\bv xc}))\] maximized subject to \eqref{EQ1}.
Here, $X(y,H_i^{\bv xc})$ is the (random) edge-color pair
 containing vertex $y$ in a uniformly random rainbow matching of
 $H_i^{\bv xc}$. Then  
$$\wiofp{\bv}{x}{c} \ge 2\med_u\wiofp{\bv}{u}{c} \ge 2\wiofp{\bv}{y}{c}.$$

We have, using \eqref{10} and \eqref{c1} and assuming $\cA_i$ that 
\beq{m1}
\log\F_i > (k-1)n\log n + n\log p_i - (c_1+1)n.
\eeq
$\F(H_i^{\bv xc})$ is the number of rainbow matchings of $H_i^{\bv xc}$.  
So,
 \beq{m4}
\log \F(H_i^{\bv xc}) = \log\wiofp{\bv}{x}{c} \ge (k-1)n\log n +
 n\log p_i - (c_1+2)n
\eeq
(by the assumption about $\bv,x,c$ and
 the failure of $\cC_i$, including $\wiofp{\bv}{x}{c} \ge
 \F_i/((2n)^{k})$).

Now a rainbow matching of $H_i^{\bv xc}$ is determined by the $\set{X(z,H_i^{\bv xc}):z\neq x}$. Let $M$ denote a uniform random rainbow matching of $H_i^{\bv xc}$. Sub-additivity of entropy then implies that
\beq{subadd}
H(M)=\log\F(H_i^{\bv xc}) \le \sum_{z\in U_k \sm\set{x}}h(z,H_i^{\bv xc}).
\eeq
 %This follows from Lemma \ref{shearer} by taking $S$ to be the set of
 %edge-color pairs of $H_i^{\bv xc}$ and the $A_z$ to be the set of
 %edge-color pairs of $H_i^{\bv xc}$ containing a particular vertex $z$. $\cF$
 %is the collection of rainbow matchings of $H_i^{\bv xc}$. 

By our choice of $y$, we have $h(z,H_i^{\bv xc}) \le h(y,H_i^{\bv xc})$ for at least half the $z$'s in $U_k\sm \set{x}$. Also, for all $z \in U_k\sm \set{x}$, we have
 \[h(z,H_i^{\bv xc}) \le \log d_{H_i^{\bv xc}}(z)\le
 \log\of{(1+\e_1)n^{k-1}p_i}.\]
Here we use the fact that ${\cal R}_i$ holds.

So,
\beq{m5}
\log\F(H_i^{\bv xc}) \le \frac{n}{2}\brac{  h(y,H_i^{\bv xc})+\log((1+\e_1)n^{k-1}p_i)}
\eeq 
and hence by combining \eqref{m4} and \eqref{m5} we get 
\begin{align}
h(y,H_i^{\bv xc}) &\ge \frac{2}{n}\log\F(H_i^{\bv xc}) - \log\of{(1+\e_1)n^{k-1}p_i} \nonumber\\
&\ge \frac{2}{n}\of{(k-1)n\log n +n\log p_i -(c_1+2)n}-(k-1)\log n - \log p_i -\e_1\nonumber\\
&= 2(k-1)\log n + 2\log p_i - (c_1+2) - (k-1)\log n - \log p_i-\e_1\nonumber \\
&\geq \log\of{d_{H_i^{\bv xc}}(y)} - (c_1+3).\label{m8}
\end{align}
To summarise what we have proved so far: If we have $\cA_i,\cR_i$ but
not $\cC_i$ then \eqref{m8} holds.

\medskip
Now for $i=1,\ldots,k-1$, let $W_i = U_i\sm\set{v_i}$ and $W = W_1\times\cdots\times W_{k-1}$. Let $L=C\sm\set{c}$ 
and for $(\tbf{z},c') \in W\times L$, let 
$\w_i'(\tbf{z},c')$ be the number of rainbow matchings of $H_i
-\set{\bv,\tbf{z},x,y}-\set{c,c'}$. We define $\w_y(\of{\tbf{z}, y},c')$ on 
\[W_y := \set{(\of{\tbf{z}, y},c') \,:\, \tbf{z}\in W\, ,\,c'\in L,\, \of{\tbf{z}, y}\in E_i,\,\i(\of{\tbf{z},y}) = c'}\]
 as $\w'_i(\tbf{z},c')$ and define 
 $\w_x(\of{\tbf{z}, x},c')$ on 
\[W_x := \set{(\of{\tbf{z}, x},c') \,:\, \tbf{z}\in W\, ,\,c'\in L,\, \of{\tbf{z}, x}\in E_i,\,\i(\of{\tbf{z},x}) = c'}\]
 as $\w'_i(\tbf{z},c')$. Then the random variable $X(y, H_i^{\bv xc})$,
 which is the edge-color pair containing $y$ in a random rainbow matching of $H_i^{\bv xc}$, is chosen according to
 $\w_y$ and $X(x,H_i^{\bv yc})$  which is the edge-color pair containing $x$ in a random rainbow matching of $H_i^{\bv yc}$, is chosen according to $\w_x$.
 
 Equation \eqref{m8} tells us that $H(X(y,H_i^{\bv xc}))=h(y,H_i^{\bv xc})  \ge
 \log\abs{W_y} - (c_1+3)$. 
We may therefore apply Lemma \ref{JKV6.2} to conclude
that there exist $a\le b\le \r a,\,\r=\r_{c_1+3}$ and a set $J\seq W_y$ with $\abs{J}
\geq \s\abs{W_y} \geq (1-\e_1)\s n^{k-1}p_i,\,\s=\s_{c_1+3}$ such that $\w_y(J) \ge 0.7 \w_y(W_y)$ and
$J = \w_y^{-1}([a,b]).$

We also let $J' := \w_{x}^{-1}([a,b])$ and
note that 
\[\w_x(J') \le \w_x(W_x)=\wiofp{\bv}{y}{c} \le .5\wiofp{\bv}{x}{c}\] while on the other hand
\beq{hand}
\w_y(J) \ge 0.7\wiofp{\bv}{x}{c}\geq 1.4\w_x(J').
\eeq
We will condition on $H_i[V\sm\set{\bv,x,y}]$ and denote the conditioning
by $\cE_1$ i.e. we will fix the edges and edge colors of this subgraph
of $H_i$. 

Next enumerate 
$$\set{(\of{\tbf{z},y},c'):\F(H_i-\set{\bv,\tbf{z},x,y}-\set{c,c'})\in
  [a,b]}=\set{(\of{\tbf{z}_j,y},c_j),\,j=1,2,\ldots,\La}.$$ 

\begin{remark}\label{rem1}
At this point we have a small technical problem. To estimate a probability below, we need to drop the conditioning $\cA_i\cR_i\bar{\cC}_i$ and then later compensate by inflating our estimates by $1/\Pr{\cA_i\cR_i\bar{\cC}_i}$. The existence of $a,b$ depends on this conditioning and we need to deal with this fact. We tackle this as we did in Section \ref{secrcb} with respect to $\ell$ and $\La$. So we will consider pairs of integers $1\leq \l\leq \m\leq \l+\log_2\r\leq 2n^2$. Then for some pair $\l,\m$ we will find $2^\l\leq a\leq b\leq 2^\m$. It is legitimate in the argument to replace $a$ by $2^\l$ and $b$ by $2^\m$ and in the analysis below consider $a,b$ as fixed, independent of $H_i$. We can then inflate our estimates of probabilities by $O(n^2)$ to account for the number of possible choices for $\l,\m$. 
\end{remark}
We define the events
$$\cD_{e,\d}=\set{e\in
  E_i,\,\i(e) = \d}.$$
For the moment replace $H_i$ by $HP_{n,k,p_i}$.
We note that the event $\F(H_i-\set{\bv,\tbf{z}_j,x,y}-\set{c,c_j})\in[a,b]$
does not depend on the occurrence or otherwise of
$\cD_{(\tbf{z}_j,y),c_j}$ for any $k$. Hence, given\\
$\set{((\tbf{z}_j,y),c_j),\,j=1,2,\ldots,\La}$ we find that without
conditioning on $\cA_i\cR_i\bar{\cC}_i$, $|J|$ is  distributed as the sum of
independent Bernoulli random variables, as in \eqref{event}.
Note also that $\cR_i$ implies that $|W_y|\geq (1-\e_1)n^{k-1}p_i$. We can assume that $\Pr{\cA_i\cR_i\bar{\cC}_i}\geq
n^{-K^{1/3}/4}$, else we have proved \eqref{arc} by default. (We have extra conditioning $\cE_1$, but this is
independent of the $\cD_{e,\d}$).  Therefore, using Fact \ref{chernoff}, 
$$1=\Pr{|J|\geq (1-\e_1)\s n^{k-1}p_i\mid \cA_i\cR_i\bar{\cC}_i\cE_1}\leq
n^{K^{1/3}/4}\bfrac{2e\La}{(1-\e_1)\s N}^{(1-\e_1)\s n^{k-1}p_i}.$$

It follows that for $K$ sufficiently large, we have
\beq{p1}
\La\geq \frac{\s N}{10}.
\eeq
Then let
$$\G_j=H_i-\set{\bv,\tbf{z}_j,x,y}-\set{c,c_j}.$$
Note that the $\F(\G_j)=\w_i'(\tbf{z}_j,c_j)$ are
completely determined by the conditioning $\cE_1$.

Then   let
\begin{align}
\w_y(J)&=\sum_{\tbf{z}\in W}1_{\set{\tbf{z},y}\in E_i} \sum_{j:\tbf{z}_j=\tbf{z}}\F(\G_j)\cdot 1_{\i({(\tbf{z}_j, y))=c_j}}\label{p10}\\
\w_x(J')&=\sum_{\tbf{z}\in W}1_{\set{\tbf{z},x}\in
  E_i}\sum_{j:\tbf{z}_j=\tbf{z}}\F(\G_j)\cdot 1_{\i({\set{\tbf{z}_j, x})=c_j}}\label{p20}
\end{align}
Let 
\begin{align*}
X_{\tbf{z}}&=\sum_{j:\tbf{z}_j=\tbf{z}}\F(\G_j)\cdot 1_{\i({\set{\tbf{z}_j, x}=c_j}}.\\
Y_{\tbf{z}}&=\sum_{j:\tbf{z}_j=\tbf{z}}\F(\G_j)\cdot 1_{\i({(\tbf{z}_j, y)=c_j}}.
\end{align*}
Note that $X_\tbf{z},Y_\tbf{z}\leq b$.

We have
\begin{align*}
Z_y&=\frac{\w_y(J)}{b}=\sum_{\tbf{z}\in W}1_{\set{\tbf{z},y}\in E_i}\frac{Y_\tbf{z}}{b}\\ Z_x&=\frac{\w_x(J')}{b}=\sum_{\tbf{z}\in W}1_{\set{\tbf{z},x}\in E_i}\frac{X_\tbf{z}}{b}
\end{align*}

It follows directly
from the expressions \eqref{p10}, \eqref{p20} that $Z_y$ and
$Z_x$ are both equal to the sum of (conditionally) independent random
variables, each bounded between 0 and 1. Furthermore, we see from
\eqref{p10}, \eqref{p20} that
\beq{p30}
\E{Z_y\mid \cE_1}=\E{Z_x\mid \cE_1}.
\eeq

What we have to show now is that we can assume that this (conditional) expectation is large.

Let 
$$L_\bz=\set{j:\bz_j=\bz}$$
and
$$W'=\set{\bz\in W:|L_\bz|\geq \g n}$$
where $\g=\s /20$.

Note that 
$$\bz\in W'\text{ implies that }\E{Y_\bz\mid\cE_1}\geq a|L_{\bz}|n^{-1}\geq a\g.$$
We have
$$|L_\bz|\leq n\text{ and }\sum_\bz|L_\bz|=\La.$$
We deduce that
$$|W'|n+\g n(n^{k-1}-|W'|)\geq \La\geq \frac{\s N}{10}.$$
Therefore 
$$|W'|\geq \frac{\s -10\g}{10(1-\g)}n^{k-1}\geq \frac{\s  n^{k-1}}{20}.$$
Hence,
$$\E{Z_y\mid \cE_1}\geq |W'|p_i\times \frac{a\g}{b}\geq \frac{K\s \log n}{20\r }.$$

Now, Hoeffding's theorem implies concentration of $Z_y$ around its (conditional) mean i.e. for arbitrarily small constant $\e$ and for large enough $K$,
$$\Pr{|Z_y-\E{Z_y\mid\cE_1}|\geq \e\E{Z_y\mid\cE_1}\mid \cE_1}\leq
n^{-dK},$$
for some $d=d(k)$.
 
The same holds for $Z_x$.
But this together with \eqref{p30} contradicts \eqref{hand}.
This completes the proof of Case 1 of \eqref{arc}. We should of course multiply all probability upper by bounds by $O(n^2)$ to account for Remark \ref{rem1}, and there is ample room for this.
\subsubsection{Case 2} 
Suppose that $\cC_i$ fails because there are vertices $\bv = (v_1,\ldots,v_k) \in \cV_k$ such that
\[ \max_{d\in C}\w_i\of{\bv,d} > \max\set{\frac{\F_i}{(2n)^k},\, 2\med_{d\in C}\w_i\of{\bv,d}}. \]
 Let $c$ be the color that maximizes $\w_i\of{\bv,d}$. Let $c^* \in  C\sm \set{c}$ be a color with $\w_i\of{\bv,c^*} \le \med_{c}\w_i\of{\bv,c}$ and 
 \[h(c^*,H_i^{\bv c}):= H(X(c^*,H_i^{\bv c}))\]
 maximized subject to this constraint. Similarly to Case 1, 
 $X(c^*, H_i^{\bv c})$ denotes the edge-color pair using the color $c^*$ in a uniformly random rainbow matching of $H_i^{\bv c}$.
 Then we can show as before that  
\beq{m2}
h(c^*,H_i^{\bv c}) \ge \log\of{cd_{H_i^{\bv c}}(c^*)} - (c_1+3).
\eeq
Indeed, we have
\beq{m3}
\w_i\of{\bv,c} \ge 2\med_d\w_i\of{\bv,d} \ge 2\w_i\of{\bv,c^*}.
\eeq 
We have \eqref{m1} and so if
$\F(H_i^{\bv c})$ is the number of rainbow matchings of $H_i^{\bv c}$,
 \beq{m6}
\log \F(H_i^{\bv c}) = \log\w_i\of{\bv,c} \ge (k-1)n\log n +
 n\log p_i -(c_1+2)n 
\eeq
(by the assumption about $v,w,c$ and the failure of $\cC_i$, including $\wiof{v}{w}{c} \ge
 \F_i/((2n)^k)$).

Now, as in \eqref{subadd},
 \[\log\F(H_i^{\bv c}) \le \sum_{d\in C\sm \set{c}}h(z,H_i^{\bv d}).\]

%by taking $S$ to be the set of
%edge-color pairs of $H_i^{\bv c}$, and the $A_d$ to be the set of
%edge-color pairs of $H_i^{\bv c}$ colored with $d$. $\cF$
%denotes the collection of rainbow matchings of $H_i^{\bv c}$. Thus
%we have $(\tbf{x},d)\in S$ appearing in $A_d$.
 By our choice of $c^*$, we have $h(d,H_i^{\bv c}) \le h(c^*,H_i^{\bv c})$ for at least half the $d$'s in $C\sm \set{c}$. Also, for all $d \in C\sm \set{c}$, we have
 \[h(d,H_i^{\bv c}) \le \log cd_{H_i^{\bv c}}(d)\le \log\of{(1+\e_1)n^{k-1}p_i}.\]
So 
\beq{m7}
\log\F(H_i^{\bv c}) \le\frac{n}{2}
  h(y,H_i^{\bv c})+ \frac{n}{2}\log((1+\e_1))n^{k-1}p_i)
\eeq
and hence by combining \eqref{m6} and \eqref{m7} we get 
\eqref{m2}, just as we obtained \eqref{m8} from \eqref{m4} and \eqref{m5}.
 
Now for $i=1,\ldots,k$, we let $W_i = U_i\sm \set{v_i}$ and $W = W_1\times\cdots\times W_k$. We let $L=C\sm\set{c,c^*}$ 
and for $\bz = (z_1,\ldots,z_k) \in W$, let 
$\w_i'(\bz)$ be the number of rainbow matchings of $H_i - \set{\bv,\bz}$ which do not use $c^*$ or $c$. Then define $\w_{c^*}(\bz)$ on 
\[W_{c^*} := \set{\bz\in W:\  \bz\in E_i,\,\i(\bz) = c^*}\]
 as $\w'_i(\bz)$ and define $\w_c(\bz)$ on 
 \[W_c:=\set{\bz\in W:\ \bz\in E_i,\,\i(\bz) = c}\]
 as $\w'_i(\bz)$. Then the random variable $X_{c^*}=X(c^*, H_i^{\bv c})$,
 which is the edge of color $c^*$ in a random rainbow matching of $H_i^{\bv c}$, is chosen according to $\w_{c^*}$ and $X_c=X(c,H_i^{\bv c^*})$  which is the edge of color $c$ in a random rainbow matching of $H_i^{\bv c^*}$ is chosen according to $\w_c$.
 
Equation \eqref{m2} tells us that $H(X_{c^*})  \ge
 \log\abs{W_{c^*}} - (c_1+3)$. Therefore we may apply Lemma \ref{JKV6.2} to conclude
that there exist $\a\le \b\le \r \a$ and a set $J\seq W_{c^*}$ with $\abs{J} \ge
\s \abs{W_{c^*}} \geq (1-\e_1)\s n^{k-1}p_i$ such that $\w_{c^*}(J) \ge 0.7 \w_{c^*}(W_{c^*})=0.7\w_i\of{\bv,c}$ and
$J = \w_{c^*}^{-1}([\a,\b]).$ We also let $J' := \w_c^{-1}([\a,\b])$ and
note that 
\[\w_c(J') \le \w_c(W_c)=\w_i\of{\bv,c^*} \le .5\w_i\of{\bv,c}\] while on the other hand
\beq{hand0}
\w_{c^*}(J) \ge 0.7\w_i\of{\bv,c}\geq 1.4\w_c(J').
\eeq
Now let $H_i$ denote the graph induced by the edges $\tbf{e}\in W$ for which 
$\i(\tbf{e})\neq c^*,c$. Fix $H_i$ and let $F_i=W\sm E(H_i)$. 

Next enumerate 
$$\Psi=\set{\bz\in F_i:\F(H_i-\set{\bv, \bz}-\set{c^*,c})\in
  [\a,\b]}=\set{\bz_j,\,j=1,2,\ldots,\La}.$$ 
Here we can proceed as indicated in Remark \ref{rem1} and treat $\a,\b$ as constants.

Suppose that we replace $H_i$ by $HP_{n,k,p_i}$. In this case, 
$\Psi$ is determined by $H_i$ and is independent of the events $\bz_j\in E_i,\i(\bz_j)\in\set{c,c^*}$. It follows that if we omit the conditioning $\cA_i\cR_i\bar{\cC}_i$ then $|W_{c^*}|$ is distributed as $Bin(\La,p_i/n)$. We still have the conditioning $\cA_i\cR_i\bar{\cC}_i$ but we can argue as before that \eqref{p1} holds.

Then with 
$$\G_j=H_i-\set{v,w,x_j,y_j}-\set{c,c^*}$$
(i.e. the graph induced by vertices $V\setminus \set{\bv,\bz_j}$,
not including edges of color $c,c^*$),
we have
\begin{align}
\w_{c^*}(J)&=\sum_{j=1}^\La\F(\G_j)1_{\bz_j\in
  E_i,\i(\bz_j)=c^*}\label{p1a}\\
\w_c(J')&=\sum_{j=1}^\La\F(\G_j)1_{\bz_j\in
  E_i,\i(\bz_j)=c}\label{p2a}
\end{align}
We have already observed the conditioning on $H_i$ means that the $\F(\G_j)$ are independent of the $1_{\bz_j\in  E_i},1_{\i({\bz_j)=c^*}},1_{\i({\bz_j)=c}}$. Thus we may condition on the values of the $\F(\G_j)$.

  It follows directly
from the expressions \eqref{p1a}, \eqref{p2a} that $Z_{c^*}=\w_{c^*}(J)/\b$ and
$Z_c=\w_c(J')/\b$ are both equal to the sum of independent random
variables, each bounded between $\a/\b$ and 1. Furthermore, we see from
\eqref{p1a}, \eqref{p2a} that
\beq{p3}
\E{Z_c\mid \cE_1}=\E{Z_{c^*}\mid \cE_1}.
\eeq

We can argue as before that $\La\geq \s N/10$. Then note that
$$\E{Z_{c^*}\mid \cE_1}\geq \frac{\a \La p _i}{n\b}\geq \frac{K\s \log n}{10\r }.$$

Now, Hoeffding's theorem implies concentration of $Z_{c^*}$ around its (conditional) mean i.e. for arbitrarily small constant $\e$ and for large enough $K$,
$$\Pr{|Z_{c^*}-\E{Z_{c^*}\mid\cE_1}|\geq \e\E{Z_{c^*}\mid\cE_1}\mid \cE_1}\leq
n^{-d'K},$$ 
for some $d'=d'(k)$.

The same holds for $Z_c$.
But this together with \eqref{p3} contradicts \eqref{hand}.
This completes the proof of Case 2 of \eqref{arc}, as well the proof of Theorem \ref{th1}.

\section{Proof of Theorem \ref{th2}}\label{Ham}
Janson and Wormald \cite{JW} proved the following theorem.
\begin{theorem}\label{JW}
Let $G=G_{n,2r},\,4\leq r=O(1)$ be a random $2r$-regular graph with vertex set $[n]$.
Suppose that the edges of $G$ are randomly colored with $n$ colors so that each color appears exactly $r$ times. Then w.h.p. $G$ contains a rainbow Hamilton cycle.
\end{theorem}
Suppose then that we have $G=G^{(n)}_{n,m}$ where $n=2\n$ is even and
$m=Kn\log n$ where $K$ is sufficiently large. We randomly assign an integer $\ell(e)\in \set{1,2,3,4}$ to each edge. We then randomly partition the set $[n]\times [4]$ into 8 sets $C_1,C_2,\ldots,C_8$ of size $\n$. We then partition the edges of $G$ into 8 sets $E_1,E_2,\ldots,E_8$. We place an edge $e$ into $E_i$ if $(c(e),\ell(e))\in C_i$ where $c(e)$ is the color of $e$. An edge goes into each $E_i$ with the same probability, 1/8, and so w.h.p. we find that $|E_i|\geq m/10$ for $i=1,2,\ldots,8$. If $|E_i|=m_i$ then the subgraph $H_i$ induced by $E_i$ is distributed as $G^{(\n)}_{n,m_i}$ and so we can apply Theorem \ref{th1} to argue that w.h.p. each $H_i$ contains a rainbow perfect matching $M_i$. If we let $\G=\bigcup_{i=1}^8M_i$ and drop the $\ell(e)$ part of the coloring, then it almost fits the hypothesis of Theorem \ref{JW}. It is 8-regular and each color appears exactly 4 times. We say almost, because $\G$ is in general, a multi-graph. It is however well-known, see for example Wormald \cite{W} that $\Gamma$ is contiguous to the random 8-regular graph $G_{n,8}$ and this implies Theorem \ref{th2} for the case where $n$ is even. 

When $n=2\n+1$ is odd, and $m=\om n\log n$ where $\om\to\infty$ then we proceed as follows. Let $p=m/N$ and for convenience, we work with $G=G^{(n)}_{n,p}$, an edge colored copy of $G_{n,p}$, in place of $G^{(n)}_{n,m}$. We decompose $G=\G_1\cup\G_2\cup \cdots\cup\G_{\om/K}$ where each $\HH_i$ is an almost independent copy of $G_{n,p';n}$ where $1-p=(1-p')^{\om/K}$.
The dependence will come when we insist that if an edge appears in $\HH_i$ and $\G_{i'}$ then it has the same color in both. We fix an $i$ and we choose some edge $e=\set{x,y}$ and contract it to a vertex $\xi$.
We also delete all edges of $\HH_i$ that have color $c(e)$ to obtain $\HH_i'$. 
Edges in $\HH_i'$ between vertices not including $\xi$ now occur independently with probability $p''=(n-1)p'/n$.
Edges involving $\xi$ appear with about twice this probability. Now $n-1$ is even and by making $K$ large enough, we can make the probability that any $\HH_i'$ fails to contain a rainbow Hamilton cycle $H_i$ less than $1/n$. Let $e_j=\set{\xi,z_j},j=1,2$ be the edges of $H_i$ that are incident with $\xi$. Now replace $\xi$ with $x,y$. If the edges $e_1,e_2$ are disjoint in $\HH_i$ then $H_i$ can be lifted to a rainbow Hamilton cycle in $\HH_i$. This happens with probability 1/2 and the lift successes are independent. So the probability that none of the $\HH_i$ contain a rainbow Hamilton cycle is at most $2^{-\om/K}\to 0$. This completes the proof of Theorem \ref{th2}.\\

\end{document}